\documentclass[12pt]{amsart}
\usepackage{amssymb,amsmath,amsthm,amsfonts,amsopn,url,bbm}
\usepackage{enumitem}
\usepackage{amscd}

\usepackage{verbatim,amsmath,latexsym,amssymb,amsbsy, graphicx}

\usepackage{hyperref, color}
\hypersetup{colorlinks,  citecolor=blue, linkcolor=blue, urlcolor=blue}

\theoremstyle{plain}
\newtheorem{thm}{Theorem}[section]
\newtheorem*{thmm}{Main Theorem}

\newtheorem{prop}[thm]{Proposition}

\newtheorem{cor}[thm]{Corollary}
\theoremstyle{definition}

\newtheorem*{defn*}{Definition}

\newtheorem{rmk}[thm]{Remark}

\newcommand{\ideal}[1]{\mathfrak{#1}}
\newcommand{\m}{\ideal{m}}

\newcommand{\ia}{\ideal{a}}

\newcommand{\ba}{\mathbf{a}}

\newcommand{\func}[1]{\mathrm{#1} \,}

\newcommand{\height}{\func{ht}}

\DeclareMathOperator{\Koszul}{\mathbb K}

\newcommand{\im}{\func{im}}

\newcommand{\arrow}[1]{\stackrel{#1}{\longrightarrow}}
 
\newcommand{\ra}{\rightarrow}
\newcommand{\eHK}{e_{\rm HK}}

\DeclareMathOperator{\otherlen}{\lambda}

\author{Neil Epstein}
\address{Department of Mathematical Sciences \\ George Mason University \\ Fairfax, VA  22030}
\email{nepstei2@gmu.edu}

\author{Javid Validashti}
\address{Department of Mathematics\\ Cleveland State University \\ Cleveland, OH 44115}
 \email{j.validashti@csuohio.edu}
 
\date{February 26, 2016}
\title{Hilbert-Kunz multiplicity of products of ideals}


\begin{document}
\begin{abstract}
We give bounds for the Hilbert-Kunz multiplicity of the product of two ideals, and we characterize the equality in terms of the tight closures of the ideals.  Connections are drawn with $*$-spread and with ordinary length calculations.
\end{abstract}
\maketitle

\section{Introduction}
The Hilbert-Kunz multiplicity \cite{Mon-HK} $\eHK(I)$ of a finite colength ideal $I$ of a local Noetherian ring $(R,\m)$ of prime characteristic $p>0$ is an important invariant in prime characteristic commutative algebra.  It has been extensively studied when the ideal in question is the \emph{maximal} ideal $\m$, in which case it characterizes the regularity of the ring \cite[Theorem 1.5]{WaYo-mulco} and has been used to explore other properties as well (e.g. finiteness of projective dimension in \cite{Mil-pd} and strong semistability of vector bundles in \cite{Tri-semiHK}). One of the main applications of the invariant applied to arbitrary $\m$-primary ideals is the fact that it governs their tight closures \cite{HHmain}. However, it has always been clear that in order to understand the Hilbert-Kunz multiplicity of the maximal ideal, one must understand the Hilbert-Kunz multiplicity of arbitrary $\m$-primary ideals, even if one does not care about tight closure \emph{per se} \cite{Mon-HK, Chang-HK}.  It is then natural to ask, given a pair of ideals $I, J$, what can one say about the Hilbert-Kunz multiplicity of their product?  With the exception of some work on asymptotic properties of $\eHK(I^n)$ for $n\gg0$ \cite{WaYo-HK2d, Hanes-HK, Tri-density}, it seems that the Hilbert-Kunz multiplicity of products of ideals has not been widely explored, even in well-behaved rings.  Hence, the current work serves as a first foray into this interesting area.\\

In \S\ref{sec:general}, we start by recalling some length calculations and inequalities in a general framework.  In \S\ref{sec:tc}, we specialize to the prime characteristic case and explore inequalities involving the Hilbert-Kunz multiplicities of $I$, $J$, and $IJ$, along with the \emph{$*$-spread} of $J$. Our main result (Proposition~\ref{pr:ineqHK}, Theorem~\ref{thm:eqthentc}, Theorem~\ref{thm:paramthenequiv}, Corollary \ref{eq8})  is the following.
\begin{thmm}

Let $(R,\m)$ be a quasi-unmixed excellent Noetherian local ring of characteristic $p>0$  of dimension $d\geq 2$, and let $I, J$ be $\m$-primary ideals. Then  \[
\eHK(IJ) \leq  \ell^*(J) \cdot \eHK(I) + \eHK(J).
\] 
Moreover, equality implies that $J \subseteq I^*$. The converse is true if $J$ has the same tight closure as a parameter ideal. Indeed, if $J$ has the same tight closure as a parameter ideal, then 
$$\eHK(IJ) \geq  \ell^*(J) \cdot  \eHK(I+J) + \eHK(J).$$

\end{thmm}

In Remark \ref{ex:RLRtc}, we show that the assumptions on the second ideal are necessary in the converse statements.  
In \S\ref{sec:recovery}, we revisit some interesting old results with our new perspective.  As indicated above, our computations in the prime characteristic case use the invariant of \emph{$*$-spread} \cite{nme*spread} in an essential way, which in turn allows us to recover a special case of a result of Epstein and Vraciu, but with a better bound (Proposition~\ref{pr:betterbound}).  We  also recover a Lech-like inequality of Huneke and Yao \cite{HuYao-HK} in the $F$-finite case (Proposition~\ref{pr:HY}) using our methods.
 
\section{Length inequalities in the general case}\label{sec:general}
In this section, we begin in a general setting and prove a length inequality involving colengths of ideals $I$, $J$, their product $IJ$, and the number of generators of $J$. Some results from this section may be well-known.  However, for lack of a published reference  and for convenience of the reader, we provide our own proofs in full, in preparation for the results from sections 3 and 4. Let $(R,\m)$ be a Noetherian local ring.  Let $\ba = a_1, \dotsc, a_\ell$ be a system of generators for a finite $R$-module $M$.  
Then we have an exact sequence \begin{equation}\label{seq0}
R^s \arrow{u_\ba} R^\ell \arrow{\pi_\ba} M \ra 0
\end{equation}
where $\pi_\ba$ is the map  sending each $e_j \mapsto a_j$, where $e_1, \dotsc, e_\ell$ is the canonical standard basis for the free module $R^\ell$, 
and $u_\ba$ is a matrix with cokernel $M$.  Let $K_\ba$ denote the kernel of $\pi_\ba$. This gives rise to a short exact sequence 
\[
0 \ra K_\ba \arrow{j_\ba} R^\ell \arrow{\pi_\ba} M \ra 0.
\]
Now let $I$ be an $\m$-primary ideal.  Taking the tensor product of (\ref{seq0}) with $R/I$, we have 
\[
R^s/I R^s \arrow{u_{\ba,I}} (R/I)^\ell \arrow{\pi_{\ba,I}} M/IM \ra 0.
\]
Note that $K_{\ba,I}:=\im u_{\ba,I} = ( K_\ba + I R^\ell ) / I R^\ell $, giving us the short exact sequence \begin{equation}\label{seq1'}
0 \ra K_{\ba,I} \arrow{j_{\ba,I}} (R/I)^\ell \arrow{\pi_{\ba,I}} M/IM \ra 0.
\end{equation}
Since $I$ has finite colength, and since we could have taken $\ba$ to be a minimal generating set for $M$, we could assume $\ell$ is the minimal number of generators of $M$, denoted by $\mu(M)$. Thus Sequence~\ref{seq1'} yields the length equality \begin{equation}\label{len0}
\mu(M) \cdot \lambda_R(R/I) = \lambda_R(K_{\ba,I}) + \lambda_R(M/IM).
\end{equation}
In particular, \emph{if $M$ is also an $\m$-primary ideal $J$}, then we have \begin{equation}\label{len0'}
\mu(J) \cdot \lambda_R(R/I) + \lambda_R(R/J) = \lambda_R(K_{\ba,I}) + \lambda_R(R/IJ).
\end{equation}
As $K_{\ba,I}$ has non-negative length, Equations~\ref{len0} and~\ref{len0'} give rise to the inequalities in the following result.
\begin{prop}\label{pr:ineq}
Let $(R,\m)$ be a Noetherian local ring.  Let $M$ be a finite $R$-module and $I$ an $\m$-primary ideal.  Then 
\begin{equation}\label{ineq0}
 \lambda_R(M/IM) \leq \mu(M) \cdot \lambda_R(R/I).
\end{equation}
If $M$ is also an $\m$-primary ideal $J$, then  
\begin{equation}\label{ineq0'}
 \lambda_R(R/IJ) \leq \mu(J) \cdot \lambda_R(R/I) + \lambda_R(R/J).
\end{equation}
\end{prop}

By induction, one obtains immediately the following corollary.

\begin{cor}\label{cor:power}
Let $(R,\m)$ be a Noetherian local ring, and $I$ an $\m$-primary ideal minimally generated by $\ell$ elements.  Then 
\[ \lambda_R(R/I^n) \leq \left(1+\ell+\cdots +\ell^{n-1} \right) \cdot  \lambda_R(R/I). \]
\end{cor}

Now, \emph{assume $M$ is an ideal $J$}, and let $(\Koszul_{\bullet}(\ba), \partial_{\bullet})$ be the Koszul complex on the sequence $\ba = a_1, \dotsc, a_\ell$.  Note that $\partial_1 = i \circ \pi_\ba$, where $i: J \hookrightarrow R$ is the natural inclusion.  Since $\Koszul_{\bullet}(\ba)$ is
a complex, we have $\im \partial_2 \subseteq \ker \partial_1 = \ker \pi_\ba = K_\ba$, with \emph{equality} if and only if $\ba$ is a regular sequence \cite[Corollary 1.6.19]{BH}. Therefore, equality holds in (\ref{ineq0'}) if and only if $K_{\ba,I} = 0$, i.e. $K_\ba \subseteq I R^\ell$, which implies that $\im \partial_2 \subseteq I R^\ell$, and all these conditions are equivalent if $\ba$ is a regular sequence.
Now \emph{assume further that $\ell \geq 2$}  (that is, $J$ is not principal).  Then $\im \partial_2$ is a sum of cyclic modules of the form $C_{ij} := R\cdot v_{ij}(\ba)$ for every pair $(i,j)$ with $1 \leq i < j \leq \ell$, where \[
v_{ij}(\ba) := -a_j e_i + a_i e_j,
\]
recalling that $e_1, \dotsc, e_\ell$ is the canonical basis for $R^\ell$ as a free $R$-module.  So if equality holds in (\ref{ineq0'}), we have 
$$-a_j e_i + a_i e_j \in C_{ij} \subseteq \im \partial_2 \subseteq I R^\ell = \oplus_{h=1}^\ell I e_h,$$
so that $a_i, a_j \in I$, which in turn implies that all $a_i \in I$, so that $J \subseteq I$. Conversely, if $\ba$ is a regular sequence and $J \subseteq I$, then each $C_{ij} \subseteq I R^\ell$, so that \[
K_\ba = \ker \partial_1 = \im \partial_2 = \sum_{i<j} C_{ij} \subseteq I R^\ell,
\]
whence $K_{\ba, I} = 0$, which means that equality holds in (\ref{ineq0'}).  Combining all this together, we have proved the following result.
\begin{thm}\label{thm:eqconds}
Let $(R,\m)$ be a Noetherian local ring.  Let $J$ be a non-principal proper ideal and $I$ an $\m$-primary ideal.  Then the equality 
$$\lambda_R(R/IJ)=\mu(J) \cdot \lambda_R(R/I) + \lambda_R(R/J)$$
implies that $J \subseteq I$.  The converse holds
if $J$ is generated by a regular sequence.
\end{thm}

We also have the following characterization of the equality in (\ref{ineq0}).
\begin{prop}\label{pr:freeness}
$ \lambda_R(M/IM)= \mu(M) \cdot \lambda_R(R/I)$ if and only if $M/IM$ is $(R/I)$-free.
\end{prop}

\begin{proof}
We may assume  (\ref{seq0}) is part of a minimal free resolution of $M$.  Then we have $K_\ba \subseteq \m R^\ell$, and hence $K_{\ba,I} \subseteq \m\cdot (R/I)^\ell$.  If $M/IM$ is a free $(R/I)$-module, then in particular it is a projective $(R/I)$-module, so the short exact sequence (\ref{seq1'}) splits.  In particular,  there is an $(R/I)$-linear map $p: (R/I)^\ell \ra K_{\ba,I}$ such that $p \circ j_{\ba,I}$ is the identity map on $K_{\ba,I}$.  Hence, \[
K_{\ba,I} = p(j_{\ba,I}(K_{\ba,I})) \subseteq p(\m \cdot (R/I)^\ell) = \m p((R/I)^\ell) = \m \cdot \im p \subseteq \m K_{\ba,I}.
\]
Then $K_{\ba,I}=0$ by the Nakayama's lemma, which means that equality holds in (\ref{ineq0}).
\end{proof}

\begin{rmk}\label{ex:RLR}
The regular sequence condition is necessary to get the converse in Theorem~\ref{thm:eqconds}. For instance,
let $(R,\m)$ be a regular local ring, and let $J$ be any $\m$-primary ideal which is not generated by an $R$-sequence.  Since the projective dimension of $R/J$ over $R$ is finite, it follows from a result of Vasconcelos \cite[Corollary 1]{Vasc-Rseq} that  $J/J^2$ is not free over $R/J$.  Then by Proposition~\ref{pr:freeness}, equality does not hold in (\ref{ineq0'}) when we let $I=J$.
\end{rmk}

\begin{rmk}
The condition $\mu(J) \geq 2$ is essential in Theorem~\ref{thm:eqconds}.  Indeed, suppose $J=(a)$ is principal.  We have $K_{a,I} = ((0:a) + I)/I$, so that equality holds in (\ref{ineq0}) if and only if $(0:a) \subseteq I$.  In particular, if $a$ is \emph{any} $R$-regular element, whether $a$ belongs to $I$ or not, we have equality in (\ref{ineq0}).
Of course, if $\mu(J)=0$, then equality holds in (\ref{ineq0}) independently of $I$, as both sides vanish.
\end{rmk}

The following result is an immediate corollary of Theorem~\ref{thm:eqconds}.

\begin{cor}\label{cor:square}
Let $R$ be a Cohen-Macaulay local ring of dimension $d\geq 2$, and let $J$ be an $\m$-primary parameter ideal.  Then \[
\lambda_R(R/J^2) = (d+1) \cdot  \lambda_R(R/J).
\]
\end{cor}

\begin{proof}
Note that $J$ is generated by a regular sequence of length $d$.  Then letting $I=J$ and $\ell = d$ in Theorem~\ref{thm:eqconds} gives the result.
\end{proof}

\section{Tight closure and Hilbert-Kunz multiplicity}\label{sec:tc}
In this section, we find analogues of the results of the previous section in prime characteristic, which allows us to look at ideals ``up to tight closure,'' replacing colength with Hilbert-Kunz multiplicity, and replacing minimal number of generators with $*$-spread. For background and unexplained terminology on tight closure theory, see the monograph \cite{HuTC}. In order that tight closure and Hilbert-Kunz multiplicity are well behaved, we make the following blanket assumptions: 
 $(R,\m)$ is an excellent $d$-dimensional Noetherian local ring of prime characteristic $p>0$,
and the $\m$-adic completion of $R$ is reduced and equidimensional -- \emph{i.e.} $R$ is \emph{quasi-unmixed}.
Let $q=p^e$ be a varying power of $p$.   Recall that for an $\m$-primary ideal $\ia$ in such a ring, the $q$th \emph{bracket power} $\ia^{[q]}$ is defined as the ideal generated by the $q$th powers of all the elements of $\ia$. It may also be defined by choosing a generating set for $\ia$ and raising these generators to $q$th powers.  The \emph{Hilbert-Kunz} multiplicity of such an ideal (which exists by \cite{Mon-HK})  is then given by \[
\eHK(\ia) := \lim_{q \rightarrow \infty} \frac{\lambda_R(R/\ia^{[q]})}{q^d}.
\]
Now let both $I$ and $J$ be $\m$-primary ideals, where $J$ is generated by a sequence $\ba = a_1, \dotsc, a_\ell$.   Replacing $I$ by $I^{[q]}$, the $a_j$ by $a_j^q$, and $M$ by $J^{[q]}$, and plugging into (\ref{len0'}), we get \[
\ell \cdot \lambda_R(R/I^{[q]}) + \lambda_R(R/J^{[q]}) = \lambda_R(K_{\ba^{q}, I^{[q]}}) + \lambda_R(R/(IJ)^{[q]}).
\]
Dividing by $q^d$ and taking the limit as $q\ra \infty$, we have \begin{equation}\label{eqHK}
\ell \cdot \eHK(I) + \eHK(J) = \lim_{q\ra \infty} \frac{\lambda_R(K_{\ba^q, I^{[q]}})}{q^d} + \eHK(IJ).
\end{equation}
Moreover, we can replace $J$ by a \emph{minimal $*$-reduction} of $J$ -- that is, an ideal contained in $J$, which has the same tight closure as $J$, and which is minimal with respect to this property.  The first named author proved in \cite[Proposition 2.1 and Lemma 2.2]{nme*spread} that under the given conditions on $R$, such an ideal always exists, and that its minimal number of generators is between $\height J$ and $\mu(J)$.  Diverging a bit from the terminology of \cite{nme*spread}, we define the \emph{$*$-spread} $\ell^*(\ia)$ of an ideal $\ia$ to be the minimum among the minimal numbers of generators of all minimal $*$-reductions of $\ia$.   To see that we can replace $J$ by an arbitrary minimal $*$-reduction (and hence one with $\ell^*(J)$ generators), use the fact that Hilbert-Kunz multiplicity is invariant up to tight closure (by \cite[Theorem 8.17]{HHmain}) and since for any $*$-reduction $K$ of $J$, we have $(IK)^* = (IJ)^*$.  Therefore, we get the following result.
\begin{prop}\label{pr:ineqHK}
Let $(R,\m)$ be a quasi-unmixed excellent Noetherian local ring of characteristic $p>0$, and let $I, J$ be $\m$-primary ideals.  Then
\begin{equation}\label{ineqHK}
\eHK(IJ) \leq \ell^*(J) \cdot \eHK(I) + \eHK(J).
\end{equation}
\end{prop}
We obtain immediately the following corollary using induction.
\begin{cor}\label{cor:powerHK}
Let $(R,\m)$ be a quasi-unmixed excellent Noetherian local ring of characteristic $p>0$, let $I$ be an $\m$-primary ideal, and let $\ell=\ell^*(I)$.  Then \[
\eHK(I^n) \leq (1+\ell+\cdots +\ell^{n-1}) \cdot \eHK(I).
\]
\end{cor}

We next give a necessary condition for equality in (\ref{ineqHK}).  Note that $\ell^*(J) \geq 2$ whenever the dimension is at least 2.
\begin{thm}\label{thm:eqthentc}
Let $(R,\m)$ be a quasi-unmixed excellent Noetherian local ring of characteristic $p>0$, and let $I, J$ be $\m$-primary ideals such that $\ell^*(J) \geq 2$. Then  the equality \[
\eHK(IJ) = \ell^*(J) \cdot \eHK(I) + \eHK(J)
\] implies that $J \subseteq I^*$.
\end{thm}

\begin{proof}
If the statement holds for some minimal $*$-reduction of $J$, then it will hold for $J$ itself.  Hence, we may pass to a minimal $*$-reduction of $J$, in which case we may let $\ell = \ell^*(J) = \mu(J)$. Let $(\Koszul_{\bullet}(\ba^q), \partial_{\bullet,q})$ be the Koszul complex on the sequence $\ba^q := a_1^q, \dotsc, a_\ell^q$, where $\ba$ is a minimal generating set for $J$.  We have $\partial_{1,q} = i \circ \pi_{\ba^q}$, where $i: J^{[q]} \hookrightarrow R$ is the natural inclusion.  Of course we have $\im \partial_{2,q} \subseteq \ker \partial_{1,q} = K_{\ba^q}$.  So suppose we have equality in (\ref{ineqHK}).  Then by Equation~\ref{eqHK}, we have \[
\lim_{q\ra \infty} \frac{\lambda_R(K_{\ba^q, I^{[q]}})}{q^d}=0.
\]
Recall from the discussion preceding Theorem~\ref{thm:eqconds} that $\im \partial_{2,q}$ is a sum of cyclic modules of the form $C_{ijq} := R \cdot v_{ij}(\ba^q)$ for every pair $i,j$ with $1 \leq i < j \leq \ell$, where $v_{ij}(\ba^q) := -a_j^q e_i + a_i^q e_j$, where $e_1, \dotsc, e_\ell$ is the canonical free basis for $R^\ell$.  Recall also that $C_{ijq} \subseteq K_{\ba^q}$.  Thus, we have 
\[
K_{\ba^q,I^{[q]}} = \frac{K_{\ba^q} + I^{[q]} R^\ell}{I^{[q]} R^\ell} \supseteq \frac{C_{ijq} + I^{[q]} R^\ell}{I^{[q]} R^\ell} 
\cong \frac{R}{I^{[q]} R^\ell :_R v_{ij}(\ba^q)} = \frac{R}{(I^{[q]} : a_i^q) \cap (I^{[q]} : a_j^q)},
\]
which maps onto $R / (I^{[q]} : a_i^q)$.
Hence, 
\[
\eHK(I) - \eHK(I+(a_i)) = \displaystyle \lim_{q\ra \infty} \frac{\lambda_R \left(\frac{I^{[q]} + (a_i^q)}{I^{[q]}}\right)}{q^d} 
= \displaystyle \lim_{q \ra \infty} \frac{\lambda_R(R / (I^{[q]} : a_i^q))}{q^d}
\leq \displaystyle  \lim_{q \ra \infty} \frac{\lambda_R(K_{\ba^q, I^{[q]}})}{q^d} = 0.
\]
By \cite[Theorem 8.17]{HHmain} then, $a_i \in I^*$.  Since this holds for all $i$, we have $J \subseteq I^*$.
\end{proof}

Next, we provide an analogue of the converse statement from Theorem~\ref{thm:eqconds}.  For this, \emph{assume $J$ is generated by a system of parameters $\ba = a_1, \dotsc, a_{d}$}, and recall that in this case, the Koszul complex $\Koszul_{\bullet}(\ba)$ is \emph{stably phantom acyclic} \cite[Proposition 5.4(d) and Proposition 5.19]{AHH}. In particular, this means that $K_{\ba^q} = \ker \partial_{1,q} \subseteq (\im \partial_{2,q})^*_{R^d}$, where for an inclusion $L \subseteq M$ of $R$-modules, $L^*_M$ denotes the tight closure of $L$ in $M$ \cite{HHmain}.  But since the entries of the matrix corresponding to $\partial_{2,q}$ are in $J^{[q]}$, we have $(\im \partial_{2,q})^*_{R^d} \subseteq (J^{[q]} R^d)^*_{R^d} = (J^{[q]})^* R^d$.  So we have $K_{\ba^q}   \subseteq (J^{[q]})^* R^d$. Therefore, 
\[
K_{\ba^q, I^{[q]}} = \frac{K_{\ba^q} + I^{[q]} R^d}{I^{[q]} R^d} \subseteq \frac{((J^{[q]})^* + I^{[q]})R^d}{I^{[q]}R^d} \subseteq \frac{((I+J)^{[q]})^* R^d}{I^{[q]}R^d}.
\]
Thus, 
\begin{align*}
\lim_{q\ra \infty} \frac{\lambda_R(K_{\ba^q, I^{[q]}})}{q^d} &\leq \lim_{q \ra \infty} \frac{\lambda_R\left({\dfrac{((I+J)^{[q]})^* R^d}{I^{[q]}R^d}}\right)}{q^d} \\
&= d \cdot \left( \lim_{q\ra \infty} \frac{\lambda_R(R/I^{[q]})}{q^d} - \lim_{q \ra \infty} \frac{\lambda_R(R / ((I+J)^{[q]})^*)}{q^d} \right).
\end{align*}
Since the latter of the two limits given is the Hilbert-Kunz multiplicity of $I+J$ when $R$ has a test element \cite[Remark 2.6]{CiuEne-tcpar}, which in turn follows from our blanket assumptions when $R$ is reduced \cite[Theorem 6.1(a)]{HHbase}, we get 
\[
\lim_{q\ra \infty} \frac{\lambda_R(K_{\ba^q, I^{[q]}})}{q^d} \leq  d \cdot \left(\eHK(I) - \eHK(I+J) \right).
\]
Combining this with Equation~\ref{eqHK}, we obtain the displayed inequality in the following result. Note that  if $J$  has the same tight closure as a parameter ideal, then $\ell^*(J) = d$.
\begin{thm}\label{thm:paramthenequiv}
Let $(R,\m)$ be a quasi-unmixed excellent reduced Noetherian local ring of characteristic $p>0$ of dimension $d\geq 2$, and let $I, J$ be $\m$-primary ideals such that $J$ has the same tight closure as a parameter ideal.  Then 
\begin{equation}\label{ineqHKpar}
\eHK(IJ) \geq d \cdot  \eHK(I+J) + \eHK(J).
\end{equation}
If $J \subseteq I^*$, then equality holds in (\ref{ineqHKpar}), in particular $\eHK(IJ) = d \cdot \eHK(I) + \eHK(J)$.
 \end{thm}

\begin{proof}
All that remains is to prove the last statement.  Since we are assuming $J \subseteq I^*$, we have $(I+J)^* = I^*$, so that $\eHK(I) = \eHK(I+J)$.  Then combine Inequality~\ref{ineqHKpar} with Inequality~\ref{ineqHK} to obtain equality.
\end{proof}

\begin{cor}\label{eq8}
Let $(R,\m)$ be a quasi-unmixed excellent reduced Noetherian local ring of characteristic $p>0$ of dimension $d\geq 2$, and let $I, J$ be $\m$-primary ideals such that $J$ has the same tight closure as a parameter ideal. 
Then $\eHK(IJ) \leq \ell^*(J) \cdot \eHK(I) + \eHK(J)$ and equality holds if and only if  $J \subseteq I^*$.
\end{cor}

\begin{cor}\label{cor:squareHK}
Let $(R,\m)$ be a quasi-unmixed reduced excellent Noetherian local ring of characteristic $p>0$ of dimension $d\geq 2$, and let $J$ be an $\m$-primary parameter ideal.  Then \[
\eHK(J^2) = (d+1) \cdot \eHK(J) = (d+1) \cdot  e(J).
\]
\end{cor}

\begin{rmk}\label{ex:RLRtc}
Let $(R,\m)$ be a regular local ring.  Since (\ref{ineq0'}) is equivalent to (\ref{ineqHK}), Remark~\ref{ex:RLR} shows that if $J$ is an $\m$-primary ideal, equality does not hold in (\ref{ineqHK}) with $I=J$ unless $J$ is generated by a regular sequence (\emph{i.e.} unless it is a parameter ideal).
\end{rmk}

\section{Revisiting some old results}\label{sec:recovery}

Recall the following theorem, stated slightly differently here than in the original paper:

\begin{thm}[Special case of {\cite[Theorem 1]{nmeVr}}]\label{thm:nmeVr}
Let $(R,\m,k)$ be an analytically irreducible excellent local ring of characteristic $p>0$ such that $k=\kappa(\bar{R})$, where $\bar{R}$ is the normalization of $R$.  Let $I, J$ be $\m$-primary ideals.  Then there exists $q_0$ such that for all $q\geq q_0$, \[
\eHK(I J^{[q]}) = \ell^*(J) \cdot \eHK(I) + \eHK(J^{[q]}).
\]
\end{thm}

In other words (since the equality $\ell^*(J) = \ell^*(J^{[q]})$ always holds), given such a ring $R$, the inequality (\ref{ineqHK}) becomes an equality \emph{when $J$ is replaced by a sufficiently high bracket power!}  We don't know how to prove this for an arbitrary quasi-unmixed reduced excellent local ring $R$, except in the case where $J$ is a parameter ideal.  However, we can recover the result from \cite{nmeVr} for such a general ring in case $J$ is a parameter ideal.  In this case, every $J^{[q]}$ is also, of course, a parameter ideal, and there exists some $q_0$ such that $J^{[q_0]} \subseteq I$, since both $I$ and $J$ are $\m$-primary.  Thus, by Theorem~\ref{thm:paramthenequiv}, we get 

\begin{prop}\label{pr:betterbound}
Let $(R,\m,k)$ be a quasi-unmixed reduced excellent Noetherian local ring of characteristic $p>0$ and dimension $d$, let $I, J$ be $\m$-primary ideals such that $J$ has the same tight closure as a system of parameters.  Let $q_0$ be a power of $p$ such that $J^{[q_0]} \subseteq I$.  Then for all $q\geq q_0$, we have \[
\eHK(I J^{[q]}) = d \cdot \eHK(I) + \eHK(J^{[q]}).
\]
\end{prop}

Finally, we re-prove the following result of Huneke and Yao, which was used in service of a proof \cite[Theorem 3.1]{HuYao-HK} that excellent quasi-unmixed local rings with Hilbert-Kunz multiplicity 1 must be regular, a theorem originally due to Watanabe and Yoshida \cite[Theorem 1.5]{WaYo-mulco}. It can also be considered as a characteristic $p>0$ analogue of Lech's inequality \cite[Theorem 3]{Lech-mult} which states that for an $\m$-primary ideal $I$ in a Noetherian local ring of dimension $d$, the Hilbert-Samuel multiplicity $e(I)$ is bounded above by $d! \cdot e(R) \cdot \lambda_R(R/I)$.  Unlike our proof, Huneke and Yao use a filtration argument.
Recall that a prime characteristic $p>0$ reduced ring $R$ is said to be \emph{$F$-finite} if the ring $R^{1/p}$ is finitely generated as a module over $R$ via the obvious inclusion map.  In this case it also follows that $R^{1/q}$ is module-finite over $R$, for all powers $q$ of $p$.

\begin{prop}[Special case of {\cite[Corollary 2.2(b)]{HuYao-HK}}]\label{pr:HY}
Let $(R,\m,k)$ be a reduced Noetherian local ring of dimension $d$ and positive characteristic $p>0$.  Suppose in addition that $R$ is $F$-finite.  Then for any $\m$-primary ideal $I$, we have \[
\eHK(I) \leq  \eHK(R) \cdot \lambda_R(R/I).
\]
\end{prop}

\begin{proof}
In Proposition~\ref{pr:ineq}, let $M=R^{1/q}$ where $q$ is a power of $p$. Therefore, 
\[
\lambda_R(M/IM)= \lambda_R (R^{1/q} / I R^{1/q}) = [k^{1/q} : k]  \cdot \otherlen_{R^{1/q}}(R^{1/q} / IR^{1/q})
= [k : k^{1/q}] \cdot  \lambda_R (R/I^{[q]}),
\]
where the notation $[k^{1/q} : k]$ denotes the field extension degree.
Also note that
\begin{align*}
\mu(M) &= \lambda_R(M/\m M) = \lambda_R (R^{1/q}/\m R^{1/q}) \\
&= [k^{1/q} :k] \cdot  \otherlen_{R^{1/q}} (R^{1/q} / \m R^{1/q}) =[k^{1/q} :k]  \cdot \lambda_R (R/\m^{[q]}).
\end{align*}
Thus by Proposition~\ref{pr:ineq}, after dividing both sides by $[k^{1/q} :k]$, we obtain
\[
\lambda_R (R/I^{[q]}) \leq \lambda_R (R/\m^{[q]}) \cdot \lambda_R(R/I).
\]
Now the result follows by dividing both sides of the above inequality by $q^d$ and taking limits as $q \to \infty$.
\end{proof}

\section*{Acknowledgment}
The authors wish to extend a warm note of gratitude to the CIRM conference center at Luminy.  Much of the work contained herein was discovered during conferences there in 2010 and 2013.

\providecommand{\bysame}{\leavevmode\hbox to3em{\hrulefill}\thinspace}
\providecommand{\MR}{\relax\ifhmode\unskip\space\fi MR }
\providecommand{\MRhref}[2]{%
  \href{http://www.ams.org/mathscinet-getitem?mr=#1}{#2}
}
\providecommand{\href}[2]{#2}

\end{document}